\documentclass[reqno]{article}
\usepackage{amsmath,amsthm,amssymb}
\usepackage{thmtools,thm-restate}
\usepackage[margin=1in]{geometry}

\usepackage{hyperref}

\theoremstyle{plain}
\newtheorem{theorem}{Theorem}[section]

\newtheorem{corollary}[theorem]{Corollary}
\newtheorem{lemma}[theorem]{Lemma}
\newtheorem{proposition}[theorem]{Proposition}
\newtheorem*{claim*}{Claim}

\theoremstyle{definition}
\newtheorem{definition}[theorem]{Definition}

\newtheorem{example}[theorem]{Example}

\title{Fair Duels after a False Start}
\author{Tanya Khovanova and Tejo Madhavarapu}
\date{\today}

\begin{document}

\maketitle

\begin{abstract}
In the fair dueling game, two players take turns shooting at each other, with the turn order defined by the following fairness criterion: the next shot is assigned to the player who is less likely to have won the duel so far. Cooper and Dutle showed that when both players have the same probability $p$ of a successful shot, the shooting sequence converges to the Thue--Morse sequence as $p \rightarrow 0$. We extend this model by forcing the sequence to begin with an arbitrary finite prefix before reverting to this fairness criterion. We prove that this initial disruption fundamentally alters the sequence's eventual behavior: in the $p \rightarrow 0$ limit, unless the initial prefix is itself a prefix of the Thue--Morse sequence or its complement, the resulting sequence is eventually periodic. Moreover, the least period is a power of 2 that is at most $4^{n+1}$, where $n$ is the length of the prefix, and the repeating block is itself a prefix of the Thue--Morse sequence.
\end{abstract}

\section{Introduction} 

In fair division problems, a common approach is to use a greedy algorithm that corrects for imbalance at each step. Choosing the order of shots in a duel is a clean, self-contained instance of this idea.
 
Cooper and Dutle studied this duel problem in their 2013 paper~\cite{CD}, calling it a \emph{greedy Galois game} after Galois, who was killed in a duel. Two players, Alice and Bob, plan a shooting sequence according to the following rule: whoever is less likely to have won (by hitting the opponent) after the first $n$ shots gets to shoot the $(n+1)^\text{st}$ shot. If there is a tie, Alice shoots next. We call this rule the fairness criterion. Here, the shooting sequence specifies the order of potential shots; the duel itself ends as soon as one player hits the other.

For example, if Alice and Bob each have a probability of $\frac{1}{3}$ of hitting, the sequence is constructed as follows:
\begin{itemize}
    \item Alice takes the first shot since the probability of either player winning in the shots preceding this shot is 0.
    \item After the first shot, Alice's probability of having won is $\frac{1}{3}$, and Bob's is 0. This is smaller than Alice's probability, so Bob shoots next.
    \item After the second shot, Alice's probability of having won is $\frac{1}{3}$, and Bob's is $\frac{2}{9}$. Therefore, Bob shoots again.
    \item After the third shot, Alice's probability of having won is $\frac{1}{3}$, and Bob's is $\frac{10}{27}$. This is greater than Alice's probability, so Alice shoots next.
\end{itemize}
By continuing this analysis, we can compute that the shooting sequence is $ABBABABBABABAABB\dots$.

Assume that Alice and Bob have the same probability of hitting their opponent with their shot, and let this probability be $p$. Cooper and Dutle \cite{CD} show that as $p \rightarrow 0$, the sequence of turns approaches the Thue--Morse sequence, whose first several terms are $ABBABAABBAABABBA$. Throughout, limits of sequences are termwise: in particular, the statement that the shooting sequence approaches a limit as $p \rightarrow 0$ means that for each $N$, there is a threshold such that every $p$ below it yields the same first $N$ shots. These prefixes agree with one another, so together they determine a single infinite shooting sequence, the limit in question. This does not necessarily mean that there is a value of $p$ that produces the whole limiting sequence.

As recounted in \cite{CD}, the Thue--Morse sequence has a remarkable history of independent discovery: Prouhet used it in 1851 to solve what is now known as the Prouhet--Tarry--Escott problem, Thue introduced it in 1906 to construct cube-free infinite words, Morse applied it to differential geometry in 1921, and Euwe used it in 1929 to show that infinitely long chess games can exist (under the rules at the time). The sequence has since appeared in combinatorics, number theory, and music; it has also arisen as a fair ordering in a sequential tournament model~\cite{PH}. For a broader survey of appearances of the Thue--Morse sequence, see~\cite{AS}.

We extend the analysis of \cite{CD}, investigating what happens in the $p \rightarrow 0$ limit if the shooting sequence starts with a specified prefix and thereafter follows the fairness criterion. This corresponds to Alice and Bob making some finite number of mistakes when computing their shooting sequence. For example, if the shooting sequence starts with $AA$ (corresponding to Alice accidentally taking one extra shot at the beginning) and follows the fairness criterion from then on, in the $p \rightarrow 0$ limit, the shooting sequence approaches the following:
\[AAB^2(BA)^4(BAAB)^{12}(BAABABBA)^{92}(ABBABAABBAABABBA)^\infty.\]
This sequence, unlike the Thue--Morse sequence, is eventually periodic! In addition, this example has period $16$, which is the smallest possible for any such periodic sequence (as we show in Section~\ref{sec:WhenIsItPeriodic}). We explain and explore this periodicity phenomenon. The shooting sequence is eventually periodic exactly when at least one of the shots after the first is computed incorrectly (with respect to the $p \rightarrow 0$ limiting case); we make this precise in the main theorem below.

\begin{restatable}{theorem}{mainthm}
\label{thm:main}
Let $s$ be a finite sequence of $A$'s and $B$'s that is not a prefix of either the Thue--Morse sequence or its complement. Consider the shooting sequence that starts with $s$ and follows the fairness criterion from then on in the $p \rightarrow 0$ limit. This shooting sequence is eventually periodic. If $s$ consists of $n$ shots, the least period is a power of $2$ that is bounded above by $4^{n+1}$. In addition, the repeating block is a prefix of the Thue--Morse sequence.
\end{restatable}

Here, the complement of a sequence of $A$'s and $B$'s is obtained by interchanging $A$'s and $B$'s, and the least period of an eventually periodic sequence is the smallest positive integer $m$ such that its $k^\text{th}$ and $(k+m)^\text{th}$ terms agree for all sufficiently large $k$.

By the result of Cooper and Dutle \cite{CD} (which we reprove near the end of Section~\ref{sec:ComputingAllTheShots}), a shooting sequence $s$ is a prefix of the Thue--Morse sequence precisely when there are no mistakes made when computing $s$ (all the shots in $s$ follow the fairness criterion), and $s$ is a prefix of the complement of the Thue--Morse sequence precisely when the first shot is the only mistake. This implies that the hypothesis of the theorem is equivalent to a mistake being made after the first shot. It also shows that the hypothesis cannot be weakened: in the excluded cases, the resulting sequence is Thue--Morse or its complement, and neither is eventually periodic by the overlap-freeness of the Thue--Morse sequence \cite{AS}.

To prove our main theorem, we begin in Section~\ref{sec:ComputingTheNextShot} by determining who shoots next.
In Section~\ref{sec:ComputingManyNextShots}, we use this rule to compute blocks of subsequent shots at once. In Section~\ref{sec:ComputingAllTheShots}, we iterate this block-by-block procedure to describe the resulting infinite sequence and explain how eventual periodicity arises. In Sections~\ref{sec:TheThueMorseMorphism} and~\ref{sec:WhenIsItPeriodic}, we develop the tools needed to determine exactly which initial prefixes lead to eventual periodicity, completing the proof of our main theorem. We conclude in Section~\ref{sec:FinalRemarks} with improved bounds, constructions, and a heuristic for the maximum possible least period.

For reference, we collect the notation used throughout the paper into Table~\ref{tab:notation}; each item is defined where it first appears.
\begin{table}[!ht]
\begin{center}
\renewcommand{\arraystretch}{1.25}
\small
\begin{tabular}{@{}l p{0.52\textwidth} l@{}}
\hline
Notation & Meaning & Introduced \\
\hline
$A$, $B$ & Alice's shot ($-1$) and Bob's shot ($+1$) & \S\ref{sec:ComputingTheNextShot} \\
$\ell(s)$ & Length of $s$: one more than the index of its last nonzero term & \S\ref{sec:ComputingTheNextShot} \\
$G_s(x)$ & $\sum_k s_k x^k$ & \S\ref{sec:ComputingTheNextShot} \\
$F_s(x)$ & Fairness polynomial $G_s(1-x)$ & \S\ref{sec:ComputingTheNextShot} \\
$c_i(s)$ & $i^\text{th}$ fairness coefficient: the coefficient of $x^i$ in $F_s(x)$ & \S\ref{sec:ComputingTheNextShot} \\
$\nu(s)$ & Fairness exponent: the least $i$ with $c_i(s) \neq 0$ & \S\ref{sec:ComputingTheNextShot} \\
$s^{\xrightarrow{k}}$ & $s$ shifted right by $k$ & \S\ref{sec:ComputingManyNextShots} \\
$T_j$, $T_\infty$ & $j^\text{th}$ Thue--Morse block, of length $2^j$; the signed Thue--Morse sequence & \S\ref{sec:ComputingManyNextShots} \\
$s \circ t$ & Concatenation: $s + t^{\xrightarrow{\ell(s)}}$ & \S\ref{sec:ComputingManyNextShots} \\
$s^n$ & $n$ concatenated copies of $s$; $s^{-n}$ means $(-s)^n$ & \S\ref{sec:ComputingManyNextShots} \\
$s^\infty$ & Infinitely many concatenated copies of $s$ & \S\ref{sec:ComputingManyNextShots} \\
$\tilde{s}$ & Infinite fair extension of $s$ & \S\ref{sec:ComputingManyNextShots} \\
$b_n^{\,s}$ & $n^\text{th}$ block count, equal to $c_n(s^{(n-1)}) / 2^{n(n-1)/2}$ & \S\ref{sec:ComputingAllTheShots} \\
$s^{(n)}$ & $n$-completion of $s$, equal to $s^{(n-1)} \circ (T_n)^{b_n^s}$ & \S\ref{sec:ComputingAllTheShots} \\
$s^\psi$ & Thue--Morse morphism: $(s^\psi)_{2n} = s_n$ and $(s^\psi)_{2n+1} = -s_n$ & \S\ref{sec:TheThueMorseMorphism} \\
\hline
\end{tabular}
\end{center}
\caption{Notation used throughout the paper}
\label{tab:notation}
\end{table}

\section{Computing the Next Shot}
\label{sec:ComputingTheNextShot}

We start by devising a way to predict the next term in a shooting sequence in the $p \rightarrow 0$ limit according to the fairness criterion. To do this, we first convert shooting sequences to integer sequences. Let a (finite or infinite) shooting sequence correspond to a unique infinite sequence of integers as follows: replace $A$ with $-1$, replace $B$ with $1$, and if the shooting sequence has ended, fill in the rest of the integer sequence with zeros.

For a nonzero integer sequence $s = (s_0,s_1, s_2, s_3, s_4, s_5, \dots)$ with finite support, let its \textit{length} be the quantity $\ell(s) = 1+\max\{k: s_k \neq 0\}$, and let $\ell(0) = 0$. For $s$ corresponding to a finite shooting sequence, let Bob's advantage after $s$ be the probability that Bob has won minus the probability that Alice has won after the shots encoded by $s$. The sign of this difference is precisely what determines who shoots next, so our first step is to describe how to compute it. This computation is Equation (1) of \cite{CD} and Lemma~1 of \cite{Shallit}; we include a proof for completeness. From here on, let a \textit{sign sequence} be an integer sequence that corresponds to a finite shooting sequence.

\begin{proposition}
\label{prop:advantage}
For a sign sequence $s$, Bob's advantage after $s$ equals $p\sum_{k=0}^{\infty}s_k(1-p)^{k}$.
\end{proposition}

\begin{proof}
    The probability that the $(k+1)^\text{st}$ shot fired is the winning shot is $p(1-p)^k$. This implies that the probability that Alice wins on one of her shots is $p\sum_{k: s_k=-1}(1-p)^{k}$, and the probability that Bob wins on one of his shots is $p\sum_{k: s_k=1}(1-p)^{k}$. This means that Bob's advantage after $s$ is the difference $p\sum_{k: s_k=1}(1-p)^{k} - p\sum_{k: s_k=-1}(1-p)^{k} = p\sum_{k=0}^{\infty}s_k(1-p)^{k}$.
\end{proof}

The fairness criterion states that Alice should shoot next if and only if Bob's advantage is nonnegative. From Proposition~\ref{prop:advantage}, we have that $\sum_{k=0}^{\infty}s_k(1-p)^{k}$ is Bob's advantage divided by $p$. If this sum is nonnegative, then Alice shoots next, and if it is negative, Bob shoots next. For small positive $p$, the sign of this sum is the same as the sign of its first nonzero coefficient when expanded as a polynomial in $p$. To make this concrete, we provide the following definition.

\begin{definition} [Fairness polynomial]
For an infinite integer sequence $s$ with finite support, the \textit{generating function} of $s$ is $G_s(x) = \sum_{n=0}^\infty s_nx^n$. The \textit{fairness polynomial} of $s$ is $F_s(x) = G_s(1-x) = \sum_{i=0}^\infty c_i(s)x^i$, where the $c_i(s)$ are the \textit{fairness coefficients} of $s$. For $s \neq 0$, the \textit{fairness exponent} of $s$ is $\nu(s) = \min\{i: c_i(s) \neq 0\}$. Let $\nu(0) = \infty$.
\end{definition}

For nonzero $s$, the generating function and fairness polynomial both have degree $\ell(s)-1$; when $s=0$, both are the zero polynomial. Thus the fairness exponent of a nonzero $s$ is at most $\ell(s)-1$.

By the observations above, we have the following corollary.

\begin{corollary}
\label{cor:nextshot}
Let $s$ be a nonzero sign sequence and let $i=\nu(s)$. For small positive $p$, Alice shoots next if $c_i(s)>0$ and Bob shoots next if $c_i(s)<0$.
\end{corollary}

To compute the fairness coefficients, we use the binomial expansion:
\[(1-x)^k = 1 - kx + \binom{k}{2}x^2 - \binom{k}{3}x^3 + \cdots.\]
This gives us
\begin{equation}
\label{eq:coefficients}
c_i(s) = (-1)^i\sum_{k=i}^{\ell(s)-1}\binom{k}{i}s_k = (-1)^i\sum_{k=i}^\infty\binom{k}{i}s_k.
\end{equation}

Here is an intuitive way to understand the first few fairness coefficients. If we let $\mathcal{A}$ be the set of turns on which Alice shoots, and $\mathcal{B}$ be the set of turns on which Bob shoots, then 
\[c_0 = \sum_{k \in \mathcal{B}} 1 - \sum_{k \in \mathcal{A}} 1, \quad c_1 = \sum_{k \in \mathcal{B}} (-k) - \sum_{k \in \mathcal{A}} (-k), \quad c_2 = \sum_{k \in \mathcal{B}} \binom{k}{2} - \sum_{k \in \mathcal{A}} \binom{k}{2}.\]
This shows how the $0^\text{th}$ fairness coefficient encapsulates the idea that it is advantageous to take more shots: for example, if Bob has taken more shots, then the $0^\text{th}$ fairness coefficient is positive, meaning that Bob is at an advantage and therefore Alice must shoot next. When the $0^\text{th}$ fairness coefficient is zero, the sum corresponding to the $1^\text{st}$ fairness coefficient encapsulates the idea that taking shots later is disadvantageous. The penalty given by this $1^\text{st}$ fairness coefficient is linear in the time of the shot, so when the players take an equal number of shots but their average shot times differ, the player with an advantage is precisely the player whose shots are earlier on average. 

Now we assume that the $0^\text{th}$ and $1^\text{st}$ fairness coefficients are both zero, corresponding to both players taking an equal number of shots and the average time of those shots being the same for both players. Since $\binom{k}{2}$ is a convex function in $k$, the $2^\text{nd}$ fairness coefficient dictates that having shot times that are spread out around the common center is more advantageous than having shots clustered near it. In fact, since $\binom{k}{2}$ is quadratic in $k$, the $2^\text{nd}$ fairness coefficient compares the variances of Bob's and Alice's shot times. Higher-order fairness coefficients capture finer distinctions between the distribution of Alice's shot times and the distribution of Bob's shot times.

We can derive the parity of the fairness coefficients easily from Equation~\eqref{eq:coefficients}. For $s$ corresponding to a shooting sequence, using the fact that $s_k$ is odd (either $1$ or $-1$) precisely when $k < \ell(s)$, we can reduce the equation modulo 2 and use the hockey stick identity to get the following: 

\begin{equation}
\label{eq:coeffmod2}
c_i(s) \equiv \sum_{k=i}^{\ell(s)-1}\binom{k}{i} \equiv \binom{\ell(s)}{i+1} \pmod{2}.
\end{equation}

Surprisingly, this parity only depends on the length of $s$, not on its specific terms.

\section{Computing Many Next Shots}
\label{sec:ComputingManyNextShots}

Now that we know how to determine the next person to shoot according to the fairness criterion, we develop an approach that allows us to calculate many of the next shots at once. Our approach hinges on the concept of \textit{Thue--Morse blocks}, which we will define formally shortly. Starting with the single term $A$, we iteratively append the complement of the current block, obtaining $AB,\, ABBA,\, ABBABAAB,\, \ldots$. As we will show in this section, these blocks and their complements describe the chunks of shots that arise when fairly extending a shooting sequence.

For integer sequences $s$ and $t$, let $s+t$ be the termwise sum of $s$ and $t$ and let $-s$ be the termwise negation of $s$. For any nonnegative integer $k$, let $s^{\xrightarrow{k}}$ be the sequence such that $(s^{\xrightarrow{k}})_n = 0$ for $0 \leq n < k$ and $(s^{\xrightarrow{k}})_n = s_{n-k}$ for $n \geq k$. This operation takes the elements of $s$ and shifts them $k$ positions to the right, so we call this sequence the \textit{$k$-shift} of $s$. One can see that
\[(-s)^{\xrightarrow{k}} = -(s^{\xrightarrow{k}}),\ \quad (s+t)^{\xrightarrow{k}} = s^{\xrightarrow{k}} + t^{\xrightarrow{k}},\ \quad \textrm{ and } \ \quad (s^{\xrightarrow{k}})^{\xrightarrow{m}} = s^{\xrightarrow{k+m}}.\]
Sums, negations, and shifts interact nicely with the generating functions we defined above, giving us the following properties:
\begin{align}
    G_{s+t}(x) = G_s(x) + G_t(x) \quad & \text{ and } \quad F_{s+t}(x) = F_s(x) + F_t(x), \nonumber \\
    G_{-s}(x) = -G_s(x) \quad & \text{ and } \quad F_{-s}(x) = -F_s(x), \nonumber \\
    G_{s^{\xrightarrow{k}}}(x) = x^kG_s(x) \quad & \text{ and } \quad F_{s^{\xrightarrow{k}}}(x) = (1-x)^kF_s(x). \label{eq:shiftgf}
\end{align}

Now we can define Thue--Morse blocks.
\begin{definition} [Thue--Morse blocks]
    Let $T_0 = (-1, 0, 0, 0, 0, \dots)$ and recursively let $T_{j+1} = T_j-(T_j)^{\xrightarrow{2^j}}$, using subtraction to mean adding the first sequence to the negation of the second. We call the sequence $T_j$ the \textit{$j^\text{th}$ Thue--Morse block}.
\end{definition}
Suppressing trailing zeros, the first few blocks and their corresponding shooting sequences are:
\begin{align*}
    T_0 &= (-1) & \longleftrightarrow \quad& A, \\
    T_1 &= (-1, 1) & \longleftrightarrow \quad& AB, \\
    T_2 &= (-1, 1, 1, -1) & \longleftrightarrow \quad& ABBA,\\
    T_3 &= (-1, 1, 1, -1, 1, -1, -1, 1) & \longleftrightarrow \quad& ABBABAAB.
\end{align*}
It can be shown that $\ell(T_j) = 2^j$ by induction.

The signed Thue--Morse sequence $T_\infty$ is defined as $\lim_{n \rightarrow \infty}T_n$. (Like the $p \rightarrow 0$ limit defined in the introduction, the limit here is taken termwise.) As we show in Section~\ref{sec:TheThueMorseMorphism}, this is equivalent to the Thue--Morse sequence as defined in \cite{AS}, with all the zeros replaced by $-1$s.

From our fairness coefficient intuition earlier, we see that $T_1$ balances the number of shots taken by Alice and Bob, so $c_0(T_1) = 0$, but the mean time of the shots is not equal, so $c_1(T_1) \neq 0$. Likewise, $T_2$ balances the number and mean time of the shots but not the variance of the times, so $c_0(T_2) = c_1(T_2) = 0$ but $c_2(T_2) \neq 0$. In a similar way, we see that $T_3$ balances the number of shots, the mean time of the shots, and the variance of the times of the shots, so $c_0(T_3) = c_1(T_3) = c_2(T_3) = 0$, while $c_3(T_3) \neq 0$ because the players' shot times differ in their asymmetry about the common mean. This pattern continues, and we will prove this below.

Consider the function $G_{T_{j+1}}(x)$. By the recursive definition of $T_j$ and the properties listed above, this function equals $(1-x^{2^j})G_{T_j}(x)$. Since $G_{T_0}(x) = -1$, we can show by induction on $j$ that \[G_{T_j}(x) = -\prod_{i=0}^{j-1}(1-x^{2^i}).\] This means that 
\begin{equation} \label{eq:fairpoly}
    F_{T_j}(x) = -\prod_{i=0}^{j-1}(1-(1-x)^{2^i}).
\end{equation}

Now, we can prove a claim about the fairness coefficients of the shifts of these Thue--Morse blocks.

\begin{lemma} 
\label{lem:coefs}
    For integers $j, k \geq 0$, the fairness exponent of $(T_j)^{\xrightarrow{k}}$ is $j$ and $c_j((T_j)^{\xrightarrow{k}}) = -2^{j(j-1)/2}$.
\end{lemma}

\begin{proof}
    By Equations \eqref{eq:shiftgf} and \eqref{eq:fairpoly}, we have that $F_{(T_j)^{\xrightarrow{k}}}(x) = -(1-x)^k\prod_{i=0}^{j-1}(1-(1-x)^{2^i})$. To find the lowest-degree term of $F_{(T_j)^{\xrightarrow{k}}}(x)$, we consider the lowest-degree term of each of its factors. The lowest-degree term of $-(1-x)^k$ is the constant term, $-1$. The constant term of the factor $1-(1-x)^{2^i}$ is 0 and its linear term is $2^ix$. Multiplying all the lowest-degree terms of the factors of $F_{(T_j)^{\xrightarrow{k}}}(x)$, we get $-1 \cdot \prod_{i=0}^{j-1}2^ix = -2^{j(j-1)/2}x^j$, so the lowest-degree term is $-2^{j(j-1)/2}x^j$, proving both claims.
\end{proof}

The vanishing of these fairness coefficients is the source of the classical connection between the Thue--Morse sequence and the Prouhet--Tarry--Escott problem mentioned in the introduction. In particular, the positions of Alice's and Bob's shots in $T_j$ have equal sums of like powers through degree $j-1$, since these powers can be expressed as a linear combination of binomial coefficients with lower index at most $j-1$.

Is $T_j$ the shortest nonzero sign sequence that has fairness exponent at least $j$? Equivalently, does $G_{T_j}$ have minimum degree among the Littlewood polynomials (polynomials with coefficients in $\{-1,1\}$) having a root of multiplicity at least $j$ at $1$? Boyd found a sequence of length $48$ with a generating function that has a multiplicity-6 root at 1~\cite{Boyd}, shorter than $T_6$, which has length $64$. We display its corresponding shooting sequence below:
\[AABBBBAB\;BAAAAABA\;ABBBAABB\;BBAABBBA\;ABAAAAAB\;BABBBBAA.\]
Thus, although Cooper and Dutle's greedy fairness rule produces the Thue--Morse blocks~\cite{CD}, those blocks need not be shortest among sign sequences with fairness exponent at least $j$. The minimum lengths for $j=1,\ldots,10$ are $2, 4, 8, 16, 32, 48, 96, 144, 192, 240$; for larger $j$, they remain unknown (see~\cite{BGPW} and the references therein). Boyd proved a lower bound of $e^{c\sqrt{j \log j}}$ for an absolute constant $c > 0$~\cite[Remark~1]{Boyd}, while known constructions give an exponential upper bound~\cite{BGPW}.

For two sequences $s$ and $t$ where $s$ has finite support, let the \emph{concatenation} of $s$ and $t$, denoted by $s \circ t$, equal $s + t^{\xrightarrow{\ell(s)}}$. This takes $t$ and places it after the last nonzero term of $s$. One basic property of concatenation is that if $t$ also has finite support, $\ell(s \circ t) = \ell(s) + \ell(t)$. Using this, one can show that concatenation is associative. Concatenation also plays nicely with negation and shifts: $-(s \circ t) = (-s) \circ (-t)$ and $(s \circ t)^{\xrightarrow{k}} = s^{\xrightarrow{k}} \circ t$ unless $s = 0$. The Thue--Morse blocks satisfy the property that $T_j \circ (-T_j) = T_{j+1}$.

Letting $s$ be an integer sequence with finite support, for a positive integer $n$, let $s^n$ be $s \circ s \circ \dots \circ s$ where the number of occurrences of $s$ is $n$. Let $s^\infty = \lim_{n \rightarrow \infty}(s^n)$; this creates a periodic sequence by concatenating $s$ to itself repeatedly. Let $s^0$ be the zero sequence, let $s^{-n} = (-s)^n$, and let $s^{-\infty} = (-s)^\infty$. The length of the sequence $s^n$ for an integer $n$ is $|n|\ell(s)$. Note that standard exponent laws fail: $s^1 \circ s^{-1} \neq s^0$ unless $s = 0$.

Let a sequence $s$ with finite support be a \emph{prefix} of another sequence $t$ if for all nonnegative $i < \ell(s)$, we have $s_i = t_i$. For any sequence $t$, we have that $s$ is a prefix of $s \circ t$. 

In preparation for the main claim of this section, we define what it means to fairly extend a shooting sequence.

\begin{definition} [Fair extensions]
    Let $s$ be an integer sequence corresponding to a finite shooting sequence $s'$. Consider the infinite continuation of $s'$ such that for any $n \geq \ell(s)$, the $(n+1)^\text{st}$ shot in the continuation follows the fairness criterion if $p$ is sufficiently small (the threshold depends on $n$). The \emph{infinite fair extension} $\tilde{s}$ of $s$ is the integer sequence corresponding to this infinite shooting sequence. A prefix of $\tilde{s}$ that has $s$ as a prefix is called a \emph{finite fair extension} of $s$.
\end{definition}

Note that if $t$ is a finite fair extension of $s$, then $\tilde{s} = \tilde{t}$ and any fair extension of $t$ is a fair extension of $s$. Corollary~\ref{cor:nextshot} provides a way to compute this sequence term-by-term, using the fairness criterion for small $p$ to identify who takes the next shot (unless $s$ is zero, in which case the first shot is taken by Alice). This computation procedure proves that $\tilde{s}$ exists and is unique.

The main claim of this section is a proposition that illustrates how Thue--Morse blocks can be used to predict several of the next shots at once. We use a lemma and its corollaries to build up to this proposition.

\begin{lemma} 
\label{lem:concatthue}
    Let $s$ be an integer sequence with finite support, and let $i$ and $j$ be nonnegative integers. If $i < j$, then $c_i(s \circ T_j) = c_i(s \circ (-T_j)) = c_i(s)$. If $i = j$, we have that $c_i(s \circ T_j) = c_i(s) - 2^{j(j-1)/2}$ and $c_i(s \circ (-T_j)) = c_i(s) + 2^{j(j-1)/2}$.
\end{lemma}

\begin{proof}
    By Lemma~\ref{lem:coefs}, we have that $\nu\big((T_j)^{\xrightarrow{\ell(s)}}\big) = j$, so the fairness coefficient $c_i\big(( T_j)^{\xrightarrow{\ell(s)}}\big)$ equals $0$ for $i < j$ and equals $-2^{j(j-1)/2}$ for $i = j$. Since $s \circ (\pm T_j) = s\pm( T_j)^{\xrightarrow{\ell(s)}}$, by the additivity of fairness coefficients, the lemma follows.
\end{proof}

Thus, $2^{j(j-1)/2}$ is the amount by which concatenating a copy of $T_j$ or $-T_j$ changes the $j^\text{th}$ fairness coefficient, while leaving all lower coefficients unchanged. From this, we derive two corollaries (using induction on $|n|$ for the second one).

\begin{corollary} 
\label{cor:fairexp}
    If the fairness exponent of a sequence $s$ is at least $j$, then $\nu(s \circ T_j) \geq j$ and $\nu(s \circ (-T_j)) \geq j$.
\end{corollary}

\begin{corollary}
\label{cor:concatmanythue}
    Let $s$ be a finitely supported integer sequence with $\nu(s) \geq j$, and let $n$ be an arbitrary integer. Then $\nu(s \circ (T_j)^n) \geq j$ and $c_j(s \circ (T_j)^n) = c_j(s) - n \cdot 2^{j(j-1)/2}$.
\end{corollary}

Now, we prove the main claim of the section, determining how these Thue--Morse blocks can be used to fairly extend a sequence.

\begin{proposition}
\label{prop:manyshots}
    For a sign sequence $s$ such that $\nu(s) \geq j$, either $s \circ T_j$ or $s \circ (-T_j)$ is a finite fair extension of $s$. If $\nu(s)=j$, then the choice is determined by the sign of $c_j(s)$:
    \[\begin{cases}
    T_j, & c_j(s)>0,\\
    -T_j, & c_j(s)<0.
    \end{cases}\]
\end{proposition}

\begin{proof}
    We prove this by induction on $j$. For the case of $j = 0$, the next shot needs to be fired by either Alice or Bob. By Corollary~\ref{cor:nextshot}, if $\nu(s) = 0$, then Alice shoots next if $c_0(s) > 0$ and Bob shoots next if $c_0(s) < 0$. This proves the base case.

    For the inductive step, assume that this proposition holds for any sign sequence $t$ with $\nu(t) \geq j$. Let $s$ be a sign sequence with $\nu(s) \geq j+1$. We have that $c_j(s) = 0$, so by Lemma~\ref{lem:concatthue}, we have $c_j(s \circ T_j) = -2^{j(j-1)/2}$ and $c_j(s \circ (-T_j)) = 2^{j(j-1)/2}$. Since Corollary~\ref{cor:fairexp} implies that $\nu(s \circ T_j) \geq j$ and $\nu(s \circ (-T_j)) \geq j$, we have $\nu(s \circ T_j) = \nu(s \circ (-T_j)) = j$. Now, by the inductive hypothesis, $s \circ T_j \circ (-T_j) = s \circ T_{j+1}$ fairly extends $s \circ T_j$ and $s \circ (-T_j) \circ T_j = s \circ (-T_{j+1})$ fairly extends $s \circ (-T_j)$. Since either $s \circ T_j$ or $s \circ (-T_j)$ is a fair extension of $s$ (by the inductive hypothesis), either $s \circ T_{j+1}$ or $s \circ (-T_{j+1})$ fairly extends $s$. 
    
    If $\nu(s) = j+1$, then $c_{j+1}(s) \neq 0$. If $c_{j+1}(s) > 0$, then Alice must take the next shot after $s$, so $s \circ (-T_{j+1})$ cannot be a fair extension of $s$. This implies that in this case, $s \circ T_{j+1}$ fairly extends $s$. Similarly, if $c_{j+1}(s) < 0$, then $s \circ (-T_{j+1})$ fairly extends $s$, proving the claim.
\end{proof}

\section{Computing All the Shots}
\label{sec:ComputingAllTheShots}

Now, we are in a position to explain and explore the phenomenon of eventual periodicity discussed in the introduction. We begin by introducing two intuitive consequences of Proposition~\ref{prop:manyshots}.

Consider a sign sequence $s$ such that $\nu(s) \geq j$. Intuitively, $s$ can be fairly extended by iteratively concatenating identical copies of $T_j$ or $-T_j$ to it (depending on the sign of $c_j(s)$), stopping before or when the $j^\text{th}$ fairness coefficient of the extended sequence reaches or crosses zero. In particular, concatenating these copies until the coefficient falls within the interval $[0, 2^{j(j-1)/2})$ yields a finite fair extension of $s$. Formalizing this stopping condition yields the following result.

\begin{corollary}
\label{cor:repeatblock}
    Let $s$ be a sign sequence and $j$ be a nonnegative integer such that $s$ has fairness exponent at least $j$. Let $q = \left\lfloor \frac{c_j(s)}{2^{j(j-1)/2}}\right\rfloor$ be the quotient and $r = c_j(s) - q \cdot 2^{j(j-1)/2}$ be the remainder of dividing $c_j(s)$ by $2^{j(j-1)/2}$, where $0 \leq r < 2^{j(j-1)/2}$. Then $s \circ (T_j)^q$ is a fair extension of $s$ whose fairness exponent is at least $j$ and whose $j^\text{th}$ fairness coefficient equals $r$.
\end{corollary}

The preceding argument can be formalized by splitting into cases depending on the sign of $c_j(s)$, inducting on the absolute value of $q$, and applying Corollary~\ref{cor:concatmanythue} and Proposition~\ref{prop:manyshots}.

Thus, as long as the remainder is zero, we can eliminate the current fairness coefficient and proceed to the next one. The first nonzero remainder, however, prevents this process from continuing and instead forces the sequence into a periodic loop. Consider what happens for a sequence $s$ with fairness exponent $j$ whose $j^\text{th}$ fairness coefficient is $r \in (0, 2^{j(j-1)/2})$. By Proposition~\ref{prop:manyshots}, Lemma~\ref{lem:concatthue}, and Corollary~\ref{cor:fairexp}, we can see that the $j^\text{th}$ fairness coefficient cycles as follows, with each arrow labeled by the block being appended:
\[r\xrightarrow{\;T_j\;}r-2^{j(j-1)/2}\xrightarrow{\;-T_j\;}r\xrightarrow{\;T_j\;}r-2^{j(j-1)/2}\xrightarrow{\;-T_j\;}r\xrightarrow{\;T_j\;}\cdots .\]
Since $T_j \circ (-T_j) = T_{j+1}$, each full cycle appends one copy of $T_{j+1}$. Thus we continue by concatenating infinitely many copies of $T_{j+1}$ to get the infinite fair extension of $s$. The following corollary formalizes this observation.

\begin{corollary}
\label{cor:infloop}
    Let $s$ be a sign sequence and $j$ be a nonnegative integer such that $s$ has fairness exponent at least $j$. If $c_j(s) \in (0, 2^{j(j-1)/2})$, then $\tilde{s} = s \circ (T_{j+1})^\infty$.
\end{corollary}

The argument above can be formalized by inductively showing that $s \circ (T_{j+1})^n$ is a prefix of $\tilde{s}$ for any nonnegative $n$. This then proves that $s \circ (T_{j+1})^\infty$ agrees with $\tilde{s}$ at every index, so it must be $\tilde{s}$.

To illustrate the power of these corollaries, we use them to compute the infinite fair extension of the sequence $s = (-1, -1, 0, 0, 0, 0, 0, \dots)$.

\begin{example} [Computing the infinite fair extension of the integer sequence corresponding to $AA$] \label{ex:AAcomputation}
    Let $s = (-1, -1, 0, 0, 0, 0, 0, \dots)$, the integer sequence corresponding to the shooting sequence $AA$. We can compute that $c_0(s) = -2$. Dividing this by $2^{0(0-1)/2} = 1$, we get a quotient of $-2$ and a remainder of 0. By Corollary~\ref{cor:repeatblock}, this means that the sequence can be fairly extended by concatenating $(T_0)^{-2}$, and the resulting sequence $s \circ (T_0)^{-2}$ has a $0^\text{th}$ fairness coefficient of 0, giving it a fairness exponent of at least $1$. The $1^\text{st}$ fairness coefficient of $s \circ (T_0)^{-2}$ is $-4$ and $\frac{-4}{2^{1 \cdot 0/2}} = -4$, so applying Corollary~\ref{cor:repeatblock} again gives us that $s \circ (T_0)^{-2} \circ (T_1)^{-4}$ is a fair extension of $s$ with fairness exponent at least $2$. Similarly, the $2^\text{nd}$ fairness coefficient of $s \circ (T_0)^{-2} \circ (T_1)^{-4}$ is $-24$ and $\frac{-24}{2^{2 \cdot 1/2}} = -12$, so $s \circ (T_0)^{-2} \circ (T_1)^{-4} \circ (T_2)^{-12}$ is a fair extension of $s$ with fairness exponent at least $3$.

    Now, something interesting happens. Computing the $3^\text{rd}$ fairness coefficient of $s \circ (T_0)^{-2} \circ (T_1)^{-4} \circ (T_2)^{-12}$ gives $-735$, and $\frac{-735}{2^{3 \cdot 2/2}} = -\frac{735}{8} = -91.875$, which is not an integer! When we apply Corollary~\ref{cor:repeatblock}, the quotient is $-92$ and the remainder is $1$, so we get that $s \circ (T_0)^{-2} \circ (T_1)^{-4} \circ (T_2)^{-12} \circ (T_3)^{-92}$ is a fair extension of $s$ with fairness exponent at least $3$ and $3^\text{rd}$ fairness coefficient equal to 1. Since $1 \in (0, 2^{3\cdot2/2})$, we can invoke Corollary~\ref{cor:infloop} to get that $s \circ (T_0)^{-2} \circ (T_1)^{-4} \circ (T_2)^{-12} \circ (T_3)^{-92}$ has infinite fair extension equal to $s \circ (T_0)^{-2} \circ (T_1)^{-4} \circ (T_2)^{-12} \circ (T_3)^{-92} \circ (T_4)^\infty$. Since this is the infinite fair extension of a fair extension of $s$, it must itself equal the infinite fair extension $\tilde{s}$ of $s$, giving us our final answer of \[\tilde{s} = s \circ (T_0)^{-2} \circ (T_1)^{-4} \circ (T_2)^{-12} \circ (T_3)^{-92} \circ (T_4)^\infty.\] This corresponds to the shooting sequence \[AAB^2(BA)^4(BAAB)^{12}(BAABABBA)^{92}(ABBABAABBAABABBA)^\infty.\]
\end{example}

We now formalize the process of successively eliminating fairness coefficients that we used in this example. Let $s^{(-1)} := s$. For nonnegative integers $n$, define the number $b_n^{\,s}$ and the sequence $s^{(n)}$ recursively as follows:

\begin{definition} [Block counts and $n$-completions] \label{def:completion} Recursively define:
    \begin{itemize}
        \item $b_n^{\,s} = \frac{c_n(s^{(n-1)})}{2^{n(n-1)/2}}$ if $s^{(n-1)}$ is defined, and it is undefined otherwise.
        \item $s^{(n)}$ is undefined if $b_n^{\,s}$ is undefined or not an integer, and is otherwise equal to $s^{(n-1)} \circ (T_n)^{b_n^{\,s}}$.
    \end{itemize}
\end{definition}

When the number $b_n^{\,s}$ is defined, we call it the $n^\text{th}$ block count of $s$. When the sequence $s^{(n)}$ is defined, we call it the $n$-completion of $s$.

For our example of $s = (-1, -1, 0, 0, 0, 0, \dots)$ above, we have:

\begin{center}
\begin{tabular}{l|l|l|l}
    $n$ & $c_n\big(s^{(n-1)}\big)$ & $b_n^{\,s}$ & $s^{(n)}$ \\
    \hline
    0 & $-2$ & $-2$ & $s \circ (T_0)^{-2}$ \\
    1 & $-4$ & $-4$ & $s \circ (T_0)^{-2} \circ (T_1)^{-4}$ \\
    2 & $-24$ & $-12$ & $s \circ (T_0)^{-2} \circ (T_1)^{-4} \circ (T_2)^{-12}$ \\
    3 & $-735$ & $-91.875$ & undefined
\end{tabular}
\end{center}

This illustrates how these block counts and completions correspond naturally to the process used to derive the infinite fair extension of $s$.

Based on Example~\ref{ex:AAcomputation}, one observation we make is that whenever the $n$-completion of $s$ is defined, it is a fair extension of $s$ and has fairness exponent greater than $n$. This turns out to be true in general, as we see in the following proposition. While reading the proofs in this section, the reader may find it helpful to refer to the example.

\begin{proposition}
\label{prop:prefix}
    Let $s$ be a nonzero sign sequence. For an integer $n \geq -1$, if $s^{(n)}$ is defined, it is a fair extension of $s$ with fairness exponent greater than $n$.
\end{proposition} 

\begin{proof}
    We shall prove this by induction on $n$. The base case of $n = -1$ is true since $s^{(-1)} = s$, which has a nonnegative fairness exponent and is a fair extension of itself.

    For the inductive step, assume that $s^{(n)}$ is defined and that $s^{(n-1)}$ is a fair extension of $s$ with fairness exponent greater than $n-1$. We shall prove that $s^{(n)}$ is a fair extension of $s$ with fairness exponent greater than $n$. Consider the block count $b_n^{\,s} = \frac{c_n(s^{(n-1)})}{2^{n(n-1)/2}}$. This must be an integer since $s^{(n)}$ is defined. This means that the quotient of the division $\frac{c_n(s^{(n-1)})}{2^{n(n-1)/2}}$ equals $b_n^{\,s}$ and the remainder equals 0. Since the fairness exponent of $s^{(n-1)}$ is at least $n$, by Corollary~\ref{cor:repeatblock}, we have that $s^{(n-1)} \circ (T_n)^{b_n^{\,s}} = s^{(n)}$ is a fair extension of $s^{(n-1)}$. In addition, $s^{(n)}$ has fairness exponent at least $n$ and $n^\text{th}$ fairness coefficient zero, so it must have fairness exponent greater than $n$. Since $s^{(n)}$ fairly extends $s^{(n-1)}$ and $s^{(n-1)}$ fairly extends $s$, we have that $s^{(n)}$ is a fair extension of $s$ with fairness exponent greater than $n$. This proves the inductive step.
\end{proof}

We are now ready to state and prove our first theorem describing the structure of $\tilde{s}$. We demonstrate that this structure differs fundamentally based on whether the $i^\text{th}$ block count of $s$ is non-integral for some nonnegative integer $i$. In the case where such an $i$ exists, the proof proceeds similarly to Example~\ref{ex:AAcomputation}.

\begin{theorem}
\label{thm:fullseq}
    Let $s$ be a nonzero sign sequence. If $s^{(n)}$ is defined for all $n$, then $\tilde{s} = \lim_{n \rightarrow \infty} s^{(n)}$. If not, let $i$ be the smallest nonnegative integer such that $s^{(i)}$ is undefined. Then \[\tilde{s}=s^{(i-1)} \circ (T_i)^{\lfloor b_i^{\,s} \rfloor} \circ (T_{i+1})^\infty.\]
\end{theorem}

\begin{proof}
    The proof of this theorem can be split into the following two cases.

    Case 1: $s^{(n)}$ is defined for all $n$. By Proposition~\ref{prop:prefix}, the sequence $s^{(n)}$ is a prefix of $\tilde{s}$ (by virtue of being a finite fair extension of $s$) and has fairness exponent greater than $n$. This implies that the degree of its fairness polynomial, which is $\ell(s^{(n)})-1$, is also greater than $n$, implying that $\ell(s^{(n)}) > n+1$. Since the sequences $s^{(n)}$ are prefixes of $\tilde{s}$ that grow arbitrarily long, their limit equals $\tilde{s}$.

    Case 2: There exists some smallest nonnegative integer $i$ such that $s^{(i)}$ is undefined. Since $s^{(i-1)}$ is defined, $b_i^{\,s}$ is defined, so it cannot be an integer (or else $s^{(i)}$ would also be defined). Since $b_i^{\,s} = \frac{c_i(s^{(i-1)})}{2^{i(i-1)/2}}$, the quotient of dividing $c_i(s^{(i-1)})$ by $2^{i(i-1)/2}$ is $\lfloor b_i^{\,s}\rfloor$ and the remainder $r$ is an integer that satisfies $0 \leq r < 2^{i(i-1)/2}$. The fact that $\frac{c_i(s^{(i-1)})}{2^{i(i-1)/2}}$ is not an integer implies that $r \neq 0$. By Proposition~\ref{prop:prefix}, we see that $s^{(i-1)}$ has fairness exponent at least $i$. Therefore, by Corollary~\ref{cor:repeatblock}, we have that $s^{(i-1)} \circ (T_i)^{\lfloor b_i^{\,s} \rfloor}$ is a fair extension of $s^{(i-1)}$ with a fairness exponent of at least $i$ and whose $i^\text{th}$ fairness coefficient is equal to $r$. By Corollary~\ref{cor:infloop}, we then get that the infinite fair extension of $s^{(i-1)} \circ (T_i)^{\lfloor b_i^{\,s} \rfloor}$ is $s^{(i-1)} \circ (T_i)^{\lfloor b_i^{\,s} \rfloor} \circ (T_{i+1})^\infty$. Since $s^{(i-1)}$ is a fair extension of $s$ by Proposition~\ref{prop:prefix}, the sequence $s^{(i-1)} \circ (T_i)^{\lfloor b_i^{\,s} \rfloor} \circ (T_{i+1})^\infty$ must equal the infinite fair extension $\tilde{s}$ of $s$, proving the theorem.
\end{proof}

Using Theorem~\ref{thm:fullseq}, we can provide an alternate proof of Cooper and Dutle's original result as follows. Let $s$ be the sequence $T_0$. We will prove the following claim, describing the completions of $s$. 

\begin{claim*}
    For all $n \geq 0$, we have $b_n^{\,s} = -1$ and $s^{(n)} = T_{n+1}$.
\end{claim*}

\begin{proof}
    We shall prove this by induction on $n$.

    Base case: $n = 0$. In this case, $b_n^{\,s} = \frac{c_n(s^{(n-1)})}{2^{n(n-1)/2}} = c_0(T_0) = -1$ and $s^{(n)} = s^{(n-1)} \circ (T_n)^{b_n^{\,s}} = T_0 \circ (-T_0) = T_1$.

    Inductive step: Assume that $s^{(n-1)} = T_{n}$. We shall prove that $b_{n}^{\,s} = -1$ and $s^{(n)} = T_{n+1}$. We have that $b_{n}^{\,s} =\frac{c_n(s^{(n-1)})}{2^{n(n-1)/2}} = \frac{c_n(T_n)}{2^{n(n-1)/2}}$, which equals $-1$ by Lemma~\ref{lem:coefs}. This gives us that the $n$-completion $s^{(n)} = s^{(n-1)} \circ (T_n)^{b_n^{\,s}} = T_n \circ (-T_n) = T_{n+1}$, proving the inductive step.
\end{proof}

From this, we derive that $\tilde{s} = \lim_{n \rightarrow \infty}s^{(n)} = \lim_{n \rightarrow \infty} T_{n+1} = T_\infty$, which corresponds to the Thue--Morse shooting sequence. This proves Cooper and Dutle's result.

One might also wonder how many completions are guaranteed to exist. We prove that every nonzero sign sequence has a 2-completion; as we illustrated above, the sequence corresponding to $AA$ does not have a 3-completion, so this bound is tight.

\begin{lemma} \label{lem:2completion}
    If $s$ is a nonzero sign sequence, then $s^{(0)}$, $s^{(1)}$, and $s^{(2)}$ are defined.
\end{lemma}

\begin{proof}
    The block count $b_0^{\,s}$ is $\frac{c_0(s)}{1}$, so $b_0^{\,s}$ is an integer, implying that $s^{(0)}$ is defined. Likewise, $b_1^{\,s}$ is $\frac{c_1(s^{(0)})}{1}$, so $s^{(1)}$ is also defined. By Proposition~\ref{prop:prefix}, the fairness exponent of $s^{(1)}$ is at least 2, so $c_0(s^{(1)}) = 0$. Since Equation \eqref{eq:coeffmod2} applied to $i=0$ implies that $c_0(s^{(1)}) \equiv \ell\big(s^{(1)}\big) \pmod{2}$, we have that $\ell\big(s^{(1)}\big) = 2m$ for some $m$. The expression $\binom{\ell(s^{(1)})}{3} = \binom{2m}{3}$ is even because it can be shown to equal $4\big(\binom{m}{3} + \binom{m+1}{3}\big)$. Equation \eqref{eq:coeffmod2} applied to $i=2$ then implies that $c_2(s^{(1)}) \equiv \binom{\ell(s^{(1)})}{3} \pmod{2}$, so $c_2(s^{(1)})$ is also even, implying that $b_2^{\,s} = \frac{c_2(s^{(1)})}{2}$ is an integer. Therefore, $s^{(2)}$ is also defined, proving the claim.
\end{proof}

\section{The Thue--Morse Morphism} \label{sec:TheThueMorseMorphism}

Now that we have described how periodicity arises, we turn to determining which shooting-sequence prefixes give rise to an eventually periodic shooting sequence, thereby proving the main theorem stated in the introduction. To do this, we first illustrate and prove another phenomenon exhibited by these shooting sequences. 

Let $s$ be the integer sequence corresponding to the shooting sequence $AA$, as in Example~\ref{ex:AAcomputation}. Consider the integer sequence $t$ corresponding to the shooting sequence $ABAB$, which is obtained from $AA$ by the substitution rule $A \rightarrow AB, B \rightarrow BA$. Computing the infinite fair extensions of $s$ and $t$ using Theorem~\ref{thm:fullseq}, we have
\begin{align*}
    \tilde{s} &= s\circ(T_0)^{-2}\circ(T_1)^{-4}\circ(T_2)^{-12}\circ(T_3)^{-92}\circ(T_4)^{\infty}, \\
    \tilde{t} &= t\circ(T_1)^{-2}\circ(T_2)^{-4}\circ(T_3)^{-12}\circ(T_4)^{-92}\circ(T_5)^{\infty}.
\end{align*}

Interestingly, the exponents of the blocks are the same and the indices of the blocks are shifted by 1. This turns out to hold in general, which we will prove in this section. Applying this substitution rule in reverse to a sequence built from $AB$ and $BA$ pairs gives a sequence of half the length. We will show that one sequence has all completions if and only if the other does, allowing us to reduce one case of the main theorem to a shorter sequence. To state this precisely, we first give this substitution rule a name.

\begin{definition} [Thue--Morse morphism]
For an integer sequence $s$, define the sequence $s^\psi$ as the integer sequence satisfying $(s^\psi)_{2n}=s_n$ and $(s^\psi)_{2n+1}=-s_n$ for every nonnegative integer $n$. Call this sequence the \emph{Thue--Morse morphism} of $s$.
\end{definition}
This coincides with the standard definition of the Thue--Morse morphism on sequences consisting of only the terms $+1$ and $-1$ and corresponds to the replacement rule $A \rightarrow AB$, $B \rightarrow BA$ that we used to derive $t$ from $s$ \cite{AS}.

The Thue--Morse morphism interacts nicely with all the sequence operations defined in Section~\ref{sec:ComputingManyNextShots}. It is a linear operator: in particular, $(s+t)^\psi = s^\psi + t^\psi$ and $(-s)^\psi=-(s^\psi)$. From the definition of length, one can also see that it doubles lengths: $\ell(s^\psi) = 2\ell(s)$. We can also see that it interacts well with the prefix relation: if $s$ is a prefix of $t$, then $s^\psi$ is a prefix of $t^\psi$. In addition, it plays nicely with shifts: $\big(s^{\xrightarrow{k}}\big)^\psi = (s^\psi)^{\xrightarrow{2k}}$. From this, using the definition of $s \circ t$ as $s + t^{\xrightarrow{\ell(s)}}$, we can get that $(s \circ t)^\psi = s^\psi \circ t^\psi$. We can then prove that for any integer $n$, the equality $\big(s^n\big)^\psi = \big(s^\psi\big)^n$ holds. (The proof proceeds by induction on $|n|$.) This implies, by taking limits, that $\big(s^\infty\big)^\psi = \big(s^\psi\big)^\infty$ and $\big(s^{-\infty}\big)^\psi = \big(s^\psi\big)^{-\infty}$. We also note that an integer sequence is a sign sequence if and only if its Thue--Morse morphism is a sign sequence.

Using the above properties, we can derive the following lemma.

\begin{lemma} \label{lem:thueexpansions}
We have that $(T_j)^\psi = T_{j+1}$.
\end{lemma}

\begin{proof}
    We prove this by induction on $j$. For the base case $j = 0$, we have that $(T_0)^\psi = (-1, 1, 0, 0, \dots) = T_1$. For the inductive step, assume that $(T_j)^\psi = T_{j+1}$. We shall prove that $(T_{j+1})^\psi = T_{j+2}$. Since the Thue--Morse morphism respects concatenation and negation and $T_{j+1}=T_j\circ(-T_j)$, we can see that $(T_{j+1})^\psi=(T_j\circ(-T_j))^\psi=(T_j)^\psi\circ(-(T_j)^\psi)=T_{j+1}\circ(-T_{j+1}) =T_{j+2}$. This completes the proof of the inductive step, proving the claim.
\end{proof}

From this, by taking limits, we see that 
\begin{equation}\label{eq:psithue}
    (T_\infty)^\psi = T_\infty.
\end{equation} This proves that our definition of the Thue--Morse sequence is equivalent to the definition of \cite{AS} with the zeros replaced by $-1$s. For context, \cite{AS} defines the Thue--Morse sequence $t$ as the unique sequence with $t_0 = 0$, $t_{2n}=t_n$ for all nonnegative $n$, and $t_{2n+1}=1-t_n$ for all nonnegative $n$. Replacing zeros with $-1$s, we get the unique sequence starting with $-1$ that is invariant under the Thue--Morse morphism.

Our main goal in this section is to generalize the observation made at the start of the section, giving a tool we use to prove the main theorem. In order to do that, we first investigate how the Thue--Morse morphism interacts with the generating functions we have defined. We start by considering $G_{s^\psi}$. By the definition of $G_{s^\psi}$, we have $G_{s^\psi}(x) = \sum_{n=0}^\infty (s^\psi)_nx^n$. Breaking this sum up by odd and even $n$, we get \[G_{s^\psi}(x) = \sum_{n=0}^\infty (s^\psi)_{2n}x^{2n} + \sum_{n=0}^\infty (s^\psi)_{2n+1}x^{2n+1}.\] Now, by using the definition of $s^\psi$, we can see that this simplifies to
\[G_{s^\psi}(x) = \sum_{n=0}^\infty s_nx^{2n} - \sum_{n=0}^\infty s_nx^{2n+1} = (1-x)\sum_{n=0}^\infty s_nx^{2n} = (1-x)G_s(x^2).\]

We now turn our attention to the polynomial $F_{s^\psi}(x)$. Substituting $1-x$ for $x$ in the identity for $G_{s^\psi}$ and using the definition of the fairness polynomial gives
\begin{equation} \label{eq:transformedfairpoly}
F_{s^\psi}(x) = xG_s((1-x)^2) = xF_s(2x-x^2).
\end{equation}

For $s \neq 0$, the lowest-degree term of $F_s(y)$ is $c_{\nu(s)}(s)\,y^{\nu(s)}$, so substituting $y = 2x-x^2$ gives 
\newline
$c_{\nu(s)}(s)\,(2x-x^2)^{\nu(s)} = c_{\nu(s)}(s)\,x^{\nu(s)}(2-x)^{\nu(s)} = 2^{\nu(s)}c_{\nu(s)}(s)\,x^{\nu(s)} + (\text{higher-degree terms})$.
Every other term $c_i(s)y^i$ of $F_s(y)$ has $i > \nu(s)$, so it contributes $c_i(s)x^i(2-x)^i$, which has no terms of degree $\leq \nu(s)$. Hence $F_s(2x-x^2)$ has lowest-degree term $2^{\nu(s)}c_{\nu(s)}(s)x^{\nu(s)}$, so $F_{s^\psi}(x) = xF_s(2x-x^2)$ has lowest-degree term $2^{\nu(s)}c_{\nu(s)}(s)x^{\nu(s)+1}$. This gives us the following two facts:
\begin{equation} \label{eq:transformedfairexp}
    \nu(s^\psi) = \nu(s) + 1,
\end{equation}
\begin{equation} \label{eq:transformedlowcoeff}
    c_{\nu(s^\psi)}(s^\psi) = 2^{\nu(s)}c_{\nu(s)}(s).
\end{equation}

These facts together imply that for nonnegative $i \leq \nu(s)$, we have
\begin{equation} \label{eq:combined}
    c_{i+1}(s^\psi) = 2^ic_{i}(s),
\end{equation}
since both sides reduce to 0 for $i < \nu(s)$ and the $i = \nu(s)$ case is Equation \eqref{eq:transformedlowcoeff}.

Note that by Equation \eqref{eq:transformedfairexp}, we have $\nu(s^\psi) > 0$, so $c_0(s^\psi) = 0$, implying that $b_0^{\,s^\psi} = 0$. We now generalize the observation made at the beginning of the section, describing how the Thue--Morse morphism interacts with completions and block counts.

\begin{proposition} \label{prop:transform}
    Let $s$ be a nonzero sign sequence. For every nonnegative integer $n$, the completion $\big(s^\psi\big)^{(n)}$ is defined if and only if $s^{(n-1)}$ is defined, in which case $\big(s^\psi\big)^{(n)} = \big(s^{(n-1)}\big)^\psi$ and $b_{n+1}^{\,s^\psi} = b_n^{\,s}$.
\end{proposition}

The idea of the proof is that $\psi$ sends $T_n$ to $T_{n+1}$ and scales the relevant fairness coefficient by $2^n$, which is exactly the factor needed to leave the block count unchanged.

\begin{proof}
We prove the claim by induction on $n$. For $n=0$, we have $s^{(-1)}=s$ and $b_0^{\,s^\psi}=0$, so we have $\big(s^\psi\big)^{(0)}= s^\psi= \big(s^{(-1)}\big)^\psi$. By Equation~\eqref{eq:combined}, we have $b_1^{\,s^\psi}= c_1(s^\psi)= c_0(s)= b_0^{\,s}$. This proves the base case.

Now let $n \geq 1$ and assume the claim holds for $n-1$. If $s^{(n-2)}$ is undefined, then $\big(s^\psi\big)^{(n-1)}$ is also undefined by the inductive hypothesis. Consequently, neither $s^{(n-1)}$ nor $\big(s^\psi\big)^{(n)}$ is defined, proving the claim in this case. Otherwise, the inductive hypothesis gives $\big(s^\psi\big)^{(n-1)} = \big(s^{(n-2)}\big)^\psi$ and $b_n^{\,s^\psi}=b_{n-1}^{\,s}$. By the recursive definition of completions, $\big(s^\psi\big)^{(n)}$ and $s^{(n-1)}$ are therefore defined precisely when this common block count is an integer. When these completions are defined, the properties of $\psi$, Lemma~\ref{lem:thueexpansions}, and the inductive hypothesis give $\big(s^\psi\big)^{(n)} = \big(s^\psi\big)^{(n-1)} \circ (T_n)^{b_n^{\,s^\psi}} = \big(s^{(n-2)}\big)^\psi \circ (T_n)^{b_{n-1}^{\,s}} = \big(s^{(n-2)} \circ (T_{n-1})^{b_{n-1}^{\,s}}\big)^\psi = \big(s^{(n-1)}\big)^\psi$.

Proposition~\ref{prop:prefix} gives $\nu\big(s^{(n-1)}\big) \geq n$, so applying Equation~\eqref{eq:combined} to $s^{(n-1)}$ and using the equality just proved, we obtain $b_{n+1}^{\,s^\psi} = \frac{c_{n+1}\big((s^\psi)^{(n)}\big)}{2^{n(n+1)/2}} = \frac{c_{n+1}\big((s^{(n-1)})^\psi\big)}{2^{n(n+1)/2}} = \frac{2^n c_n\big(s^{(n-1)}\big)}{2^{n(n+1)/2}} = \frac{c_n\big(s^{(n-1)}\big)}{2^{n(n-1)/2}} = b_n^{\,s}$. This completes the induction.
\end{proof}

From this, we derive the following corollary, which explains the phenomenon at the start of the section. This corollary can also be understood directly from Cooper and Dutle's approach: Equation~\eqref{eq:transformedfairpoly} shows that the fairness decisions for $s^\psi$ mirror those for $s$. We instead prove the result using completions, since that approach is more closely tied to the machinery developed above.

\begin{corollary} \label{cor:transform}
    For a nonzero sign sequence $s$, we have $(\tilde{s})^\psi = \widetilde{s^\psi}$.
\end{corollary}

\begin{proof}
    We split into cases depending on whether $s$ has $n$-completions for all $n$.

    Case 1: For all nonnegative $n$, the completion $s^{(n)}$ is defined. By the first half of Theorem~\ref{thm:fullseq}, this means that $\tilde{s} = \lim_{n \rightarrow \infty} s^{(n)}$, so $(\tilde{s})^\psi = \lim_{n \rightarrow \infty} \big(s^{(n)}\big)^\psi$. By Proposition~\ref{prop:transform}, the completion $\big(s^\psi\big)^{(n+1)}$ is defined for all nonnegative integers $n$ and equals $\big(s^{(n)}\big)^\psi$, so $\lim_{n \rightarrow \infty} \big(s^{(n)}\big)^\psi = \lim_{n \rightarrow \infty} \big(s^\psi\big)^{(n+1)} = \lim_{n \rightarrow \infty} \big(s^\psi\big)^{(n)}$, which equals $\widetilde{s^\psi}$ by the first half of Theorem~\ref{thm:fullseq}.

    Case 2: There exists $n$ such that $s$ does not have an $n$-completion. By Proposition~\ref{prop:transform}, this is equivalent to the $(n+1)$-completion of $s^\psi$ being undefined. Therefore, if $i$ is the smallest integer such that $s$ does not have an $i$-completion, it is also the smallest integer such that $s^\psi$ does not have an $(i+1)$-completion. By the second half of Theorem~\ref{thm:fullseq} applied to the sequences $s$ and $s^\psi$, we have
    \begin{align*}
        (\tilde{s})^\psi &= \big(s^{(i-1)} \circ (T_i)^{\lfloor b_i^{\,s} \rfloor} \circ (T_{i+1})^\infty\big)^\psi, \\
        \widetilde{s^\psi} &= \big(s^\psi\big)^{(i)} \circ (T_{i+1})^{\lfloor b_{i+1}^{\,s^\psi} \rfloor} \circ (T_{i+2})^\infty.
    \end{align*}
    By Proposition~\ref{prop:transform}, Lemma~\ref{lem:thueexpansions}, and the properties of $\psi$ given at the beginning of the section, the right-hand sides of these two equations are equal. Therefore, the left-hand sides are also equal, proving the claim.
\end{proof}

As a special case of this statement, we recover the result of Cooper and Dutle \cite{CD}, proved earlier in Section~\ref{sec:ComputingAllTheShots}. Letting $s = T_0$, we have $s^\psi = T_1$. Since $s^\psi$ is a fair extension of $s$, we have $\tilde{s}=\widetilde{s^\psi} = (\tilde{s})^\psi$. Since $\tilde{s}$ starts with $-1$ (because $s = T_0$ is a prefix of $\tilde{s}$) and is invariant under the Thue--Morse morphism, it must be the sequence $T_\infty$.

\section{When Is It Periodic?} \label{sec:WhenIsItPeriodic}

In this section, we use all of the machinery developed above in order to prove the main theorem. For convenience we restate the main theorem here.

\mainthm*

This section will build up to a proof of the following statement. This claim looks quite different from the main theorem, but Theorem~\ref{thm:fullseq} bridges the two.

\begin{theorem}
\label{thm:restated}
For a sign sequence $s$ that is not a prefix of either $T_\infty$ or $-T_\infty$, the completion $s^{(2\ell(s)+1)}$ is undefined.
\end{theorem}

Using this theorem, we can prove the main theorem as follows. Let $s$ be a shooting sequence of length $n$ that satisfies the hypothesis of the main theorem, and let $t$ be the corresponding integer sequence. By Theorem~\ref{thm:restated}, the sequence $t$ does not have a $(2n+1)$-completion, so there is a smallest nonnegative integer $i$ such that $t$ does not have an $i$-completion, and $i \leq 2n+1$. From Theorem~\ref{thm:fullseq}, we get that $\tilde{t}$ can be written as $u \circ (T_{i+1})^\infty$ for a sign sequence $u$. Since $T_{i+1}$ has length $2^{i+1}$, this must be a period of $\tilde{t}$, so the least period of $\tilde{t}$ must divide it. However, $2^i$ is not a period of $\tilde{t}$ due to the two halves of $T_{i+1}$ being complementary, so the least period does not divide $2^i$. Therefore, the least period is $2^{i+1} \leq 2^{2n+2} = 4^{n+1}$. Since $T_{i+1}$ is a prefix of the Thue--Morse sequence, the main theorem follows. 

Since $t$ was an arbitrary sequence satisfying the hypothesis of the main theorem, Theorem~\ref{thm:restated} also tells us the smallest period that a shooting sequence generated by fairly extending a prefix can have. The repeating block of $\tilde{t}$ is $T_{i+1}$, so $2^{i+1}$ is the least period of $\tilde{t}$. By Lemma~\ref{lem:2completion}, the sequence $t$ has a $2$-completion, so $i \geq 3$, implying that this period is at least $16$. The shooting sequence beginning with $AA$, computed in Section~\ref{sec:ComputingAllTheShots}, has least period exactly $16$, so no eventually periodic duel that is fair after an initial prefix can have a lower least period.

In addition, Theorem~\ref{thm:restated} proves that under its hypotheses, Theorem~\ref{thm:fullseq}'s construction always terminates and yields a closed form for the infinite fair extension. When instead $s$ is a prefix of $T_\infty$ or $-T_\infty$, its infinite fair extension is $T_\infty$ or $-T_\infty$ respectively~\cite{CD}, which is not eventually periodic~\cite{AS}; so, by Theorem~\ref{thm:fullseq}, the sequence $s^{(n)}$ is defined for every $n$. In other words, $s$ has an $n$-completion for every $n$ exactly when $s$ is a prefix of $T_\infty$ or $-T_\infty$.

To prove Theorem~\ref{thm:restated}, we first consider what happens when $s$ has a pair of terms $s_{2k} = s_{2k+1} = \pm 1$, proving a lemma showing that such a sequence cannot have a $q$-completion for a specific $q \leq 2\ell(s) + 1$. This proves one case of the theorem. In the case where no such pair exists, we show that $s$, after possibly appending one term, is equal to the Thue--Morse morphism applied to a shorter sequence; applying the theorem inductively to this shorter sequence, we get that $s$ does not have a $(2\ell(s)+1)$-completion in this case either.

Our proof of the first lemma described above will require evaluating some fairness coefficients modulo $4$, so we start by examining binomial coefficients modulo powers of 2. We use the following lemma, both of whose parts follow from Kummer's theorem \cite{Kummer, Granville}. For context, Kummer's theorem states that for a prime $p$, the largest integer $i$ such that $p^i$ divides $\binom{n}{m}$ is equal to the number of carries performed when adding $n-m$ to $m$ in base $p$. The first part of the lemma also follows from Lucas's theorem \cite{Granville}.

\begin{lemma}\label{lem:binomialfacts}
    Let $q = 2^a - 1$ for some nonnegative integer $a$ and let $n$ be a nonnegative integer. Then:
    \begin{enumerate}
        \item $\binom{n}{q}$ is odd if and only if $n \equiv q \pmod{q+1}$.
        \item If $n$ is even, then $\binom{n}{q} \equiv 0 \pmod{q+1}$.
    \end{enumerate}
\end{lemma}

Using the facts above, we can now show that if $s$ has a pair of elements $s_{2k} = s_{2k+1} = \pm 1$, then it does not have a $q$-completion, where $q = 2^a - 1$ with $2^a$ being the smallest power of 2 that is at least $\ell(s)+2$. The idea is to ``highlight" the pair $s_{2k}$, $s_{2k+1}$ using fairness coefficients. Supposing that $s^{(q)}$ exists, it has fairness exponent greater than $q$, so all its shifts have $q^\text{th}$ fairness coefficient $0$. We add the $q^\text{th}$ fairness coefficients of two of these shifts differing by one position; this sum equals 0. Expanding this out using Equation \eqref{eq:coefficients} gives a sum where each binomial coefficient is multiplied by the sum of a pair of terms, such as $s_{2k} + s_{2k+1}$. When reducing this sum modulo 4 using Lemma~\ref{lem:binomialfacts}, only the pairs whose corresponding binomial coefficients are odd survive, and these pairs are spaced $q+1$ terms apart. The first of these pairs is $s_{2k}$, $s_{2k+1}$, and $q$ is large enough that all of the other pairs occur after the initial $s^{(0)}$, forcing them to be either a $T_1$ or a $-T_1$, which makes their terms sum to 0. Therefore, the sum of these highlighted terms is $s_{2k} + s_{2k+1} = \pm2$, so the sum of fairness coefficients is congruent to $2$ modulo $4$, contradicting the fact that it equals 0.

Now we state and prove this claim formally.
\begin{lemma}
\label{lem:nobigcompletion}
    Let $s$ be a sign sequence of length $n$. Suppose for some nonnegative integer $k$, we have $s_{2k} = s_{2k+1} = \pm 1$. Let $2^a$ be the smallest power of $2$ that is at least $n+2$, and let $q = 2^a - 1$. Then $s^{(q)}$ is undefined.
\end{lemma}

\begin{proof}
    We first note that $s$ is nonempty, so $2^a \geq n+2 \geq 3$, giving us that $a \geq 2$. Let $k$ be the largest index such that $s_{2k} = s_{2k+1} = \pm 1$. (Such a largest index exists due to the hypothesis of the lemma and the fact that $s$ has finite support.) This is equivalent to $k$ being the largest nonnegative integer such that $s_{2k+1} \neq 0$ and $s_{2k+1} \neq -s_{2k}$.
    
    Consider the sequence $s^{(0)}$. Since $s^{(0)}$ has a positive fairness exponent (by Proposition~\ref{prop:prefix}), its $0^\text{th}$ fairness coefficient is zero, so $s^{(0)}$ must have equally many 1s and $-1$s. This implies that its length is even, so we write $\ell(s^{(0)}) = 2N$. Also, $s^{(0)}$ equals $s \circ (T_0)^{b_0^{\,s}} = s \circ (T_0)^{c_0(s)}$, so its length is $\ell(s) + \ell(T_0)|c_0(s)| = n + |c_0(s)|$. Since $c_0(s)$ is the sum of the terms of $s$, grouping terms into pairs, we have that \[|c_0(s)| = \left|\sum_{i=0}^{\infty}(s_{2i}+s_{2i+1})\right| \leq \sum_{i=0}^{\infty}|s_{2i}+s_{2i+1}|.\] Since $k$ is the largest nonnegative integer such that $s_{2k+1} \neq 0$ and $s_{2k+1} \neq -s_{2k}$, for every $i > k$, we have either $s_{2i+1} = 0$ or $s_{2i+1} = -s_{2i}$. If $s_{2i+1} = -s_{2i}$, then the term $|s_{2i}+s_{2i+1}|$ vanishes. If $s_{2i+1} = 0$ but $s_{2i+1} \neq -s_{2i}$, then since $s$ corresponds to a finite shooting sequence, $s_{2i+1} = 0$ implies that $n \leq 2i+1$ and $s_{2i} \neq 0$ implies that $n > 2i$, so $n=2i+1$ and $|s_{2i}+s_{2i+1}| = |s_{2i}| = 1$. In particular, $|s_{2i}+s_{2i+1}|$ is nonzero for at most one value of $i > k$, and this value gives $|s_{2i}+s_{2i+1}| = 1$. For $0 \leq i \leq k$, the term $|s_{2i}+s_{2i+1}|$ is at most 2. This implies that \[|c_0(s)| \leq \sum_{i=0}^{\infty}|s_{2i}+s_{2i+1}| \leq 2(k+1)+1 = 2k+3,\] so $2N = \ell(s^{(0)}) \leq n+2k+3$.

    After the shots taken according to the sequence $s^{(0)}$, the next 2 shots taken correspond to either the sequence $T_1$ or $-T_1$ by Proposition~\ref{prop:manyshots} and the fact that $s^{(0)}$ has a positive fairness exponent. By Corollary~\ref{cor:fairexp}, concatenating $T_1$ or $-T_1$ to $s^{(0)}$ cannot make the fairness exponent drop below 1. This means that the pair of shots taken after this also corresponds to either $T_1$ or $-T_1$, and the fairness exponent after concatenating this pair is at least 1. Continuing like this, we see that $\tilde{s}$ is $s^{(0)}$ followed by infinitely many shifted copies of either $T_1$ or $-T_1$. For all integers $m \geq N$, the $(m-N+1)^\text{st}$ shifted copy of $T_1$ or $-T_1$ occurs at terms indexed by $2m$ and $2m+1$, so 

    \begin{equation} \label{eq:pairing}
    \tilde{s}_{2m} = -\tilde{s}_{2m+1} \text{ for all } m \geq N.
    \end{equation}

    Now suppose, for contradiction, that $s^{(q)}$ is defined, and let $h = s^{(q)}$ and $M = \ell(h)$. By Proposition~\ref{prop:prefix}, the sequence $h$ has fairness exponent greater than $q$ and is a prefix of $\tilde{s}$ (due to being a finite fair extension of $s$). Since $c_0(h) = 0$, we have that $h$ contains an equal number of 1s and $-1$s, so $\ell(h) = M$ is even. We also note that $M = \ell\big(s^{(q)}\big) \geq \ell\big(s^{(0)}\big) = 2N$ because $s^{(0)}$ is a prefix of $s^{(q)}$.

    Let the sequence $h^{\xrightarrow{q-2k-1}}$ be denoted by $t$. The reason for the shift $q-2k-1$ is that it makes the contribution of the distinguished pair $h_{2k},h_{2k+1}$ particularly simple in the $q^\text{th}$ fairness coefficient, as the argument below shows. This shift is nonnegative since $q \geq n + 1 > 2k+1$. Since $F_t(x) = (1-x)^{q-2k-1}F_h(x)$, the polynomials $F_h(x)$ and $F_t(x)$ have the same lowest-degree term, so $t$ has the same fairness exponent as $h$. Similarly, $t^{\xrightarrow{1}}$ has the same fairness exponent as $h$. Since this fairness exponent is greater than $q$, we have that $c_{q}(t) = c_{q}(t^{\xrightarrow{1}}) = 0$.

    We now derive a contradiction by proving that $-\big(c_{q}(t) + c_{q}(t^{\xrightarrow{1}})\big)$ is congruent to 2 modulo 4. Since $q$ is odd, using Equation~\eqref{eq:coefficients}, we get
    \[-\big(c_{q}(t) + c_{q}(t^{\xrightarrow{1}})\big) = \sum_{i=q}^\infty\binom{i}{q}\left((t^{\xrightarrow{1}})_i+t_i\right) \equiv\sum_{g=k}^{M/2-1}\binom{2g+q-2k}{q}\bigl(h_{2g}+h_{2g+1}\bigr)\pmod 4. \]

    In the congruence, we reindex $i$ to $2g+q-2k$, making use of the fact that all terms with even $i$ in the first sum vanish modulo 4. This happens because $\binom{i}{q}$ vanishes modulo $q+1=2^a$ for even $i$ by the second part of Lemma~\ref{lem:binomialfacts} and $4 \mid 2^a$ since $a \geq 2$. Since $h$ has length $M$ and is a prefix of $\tilde{s}$ (by Proposition~\ref{prop:prefix}), we have that $h_i = \tilde{s}_i$ for $i < M$, so by Equation \eqref{eq:pairing}, we have
    \[\sum_{g=k}^{M/2-1}\binom{2g+q-2k}{q}\bigl(h_{2g}+h_{2g+1}\bigr) = \sum_{g=k}^{M/2-1}\binom{2g+q-2k}{q}\bigl(\tilde{s}_{2g}+\tilde{s}_{2g+1}\bigr) = \sum_{g=k}^{N-1}\binom{2g+q-2k}{q}\bigl(\tilde{s}_{2g}+\tilde{s}_{2g+1}\bigr). \]
    This helps us conclude that 
    \[-\big(c_{q}(t) + c_{q}(t^{\xrightarrow{1}})\big) \equiv \sum_{g=k}^{N-1}\binom{2g+q-2k}{q}\bigl(\tilde{s}_{2g}+\tilde{s}_{2g+1}\bigr) \pmod{4}.\]
    We have that $k \leq \frac{\ell(s)}{2} -1 \leq \frac{\ell(s^{(0)})}{2} - 1 = N-1$, so the sum contains the $g=k$ term.

    We now show that, modulo 4, every term of this sum except the $g = k$ term vanishes. For each term of the summation above, the second factor is even because $|\tilde{s}_{2g}| = |\tilde{s}_{2g+1}| = 1$. By the first part of Lemma~\ref{lem:binomialfacts} and the fact that $2^a = q+1$, the expression $\binom{2g+q-2k}{q}$ is odd precisely when $2g+q-2k \equiv q \pmod{2^a}$, which occurs when $g \equiv k \pmod{2^{a-1}}$. By the inequalities $2N \leq n + 2k + 3$ and $n + 2 \leq 2^a$, we have that $N-1 \leq \left\lfloor\frac{n+2k+3}{2}\right\rfloor-1 \leq \left\lfloor\frac{2^a+2k+1}{2}\right\rfloor-1 < k+2^{a-1}$. This means that the only term of the above sum that is not a multiple of 4 is the $g=k$ term, so, since $\ell(s) > 2k+1$, computing this term gives
    \[-\big(c_{q}(t) + c_{q}(t^{\xrightarrow{1}})\big) \equiv \binom{q}{q}(\tilde{s}_{2k} + \tilde{s}_{2k+1}) \equiv s_{2k} + s_{2k+1} \equiv 2s_{2k} \equiv 2 \pmod{4}.\] 
    This is a contradiction.
\end{proof}

Now we can finally prove Theorem~\ref{thm:restated}.

\begin{proof}[Proof of Theorem~\ref{thm:restated}]
    Let $n := \ell(s)$. We shall prove this by induction on $n$. The base cases of $n = 0$ and $n = 1$ are vacuously true since this implies that $s$ is a prefix of either $T_\infty$ or $-T_\infty$, so the conditions on $s$ given are not met.

    For the inductive step, assume that $n>1$ and that this theorem is true for sequences of length less than $n$. We will split the inductive step into two cases.

    Case 1: There exists a nonnegative integer $k$ such that $s_{2k} = s_{2k+1} = \pm 1$. In this case, let $2^a$ be the smallest power of 2 that is at least $n+2$, and let $q = 2^a - 1$. We have that $2^a$ is even and less than $2(n+2)$, so $2^a \leq 2(n+1)$, giving $q \leq 2n+1$. By Lemma~\ref{lem:nobigcompletion}, the completion $s^{(q)}$ is undefined, so the completion $s^{(2n+1)}$ is also undefined (since completions are defined recursively). This proves the inductive step for this case.

    Case 2: There is no nonnegative integer $k$ such that $s_{2k} = s_{2k+1} = \pm 1$. Equivalently, for every nonnegative integer $k$, either $s_{2k+1} = 0$ or $s_{2k+1} = -s_{2k}$. 

    Consider the sequence $s^{(0)}$. The number $c_0(s)$ can be written as $\sum_{i=0}^\infty (s_{2i} + s_{2i+1})$. If $n$ is even, then all the terms of this sum vanish, so $c_0(s) = 0$. If $n$ is odd, then only the term corresponding to $i = \frac{n-1}{2}$ does not vanish, so the sum, which is $c_0(s)$, equals $s_{n-1}$. We have that $s^{(0)} = s \circ (T_0)^{c_0(s)}$ by Definition~\ref{def:completion}. This means that if $n$ is even, then $s^{(0)} = s$, and if $n$ is odd, $s^{(0)}$ equals the sequence $t$ such that $t_n = -s_{n-1}$ and $t_k=s_k$ for $k \neq n$. In either case, $\ell(s^{(0)}) = 2\lceil\frac{n}{2}\rceil$ and $\big(s^{(0)}\big)_{2k+1} = -\big(s^{(0)}\big)_{2k}$ for every nonnegative $k$.

    By Proposition~\ref{prop:prefix}, we have $c_0\big(s^{(0)}\big) = 0$, so the $0$-completion of $s^{(0)}$ is $s^{(0)}$ itself, which is also the $0$-completion of $s$. Since each block count and completion depends only on the previous completion, induction gives that $s^{(i)}$ is defined if and only if $\big(s^{(0)}\big)^{(i)}$ is, and they are then equal.
    
    Let $r$ be the sequence such that $r_k = s_{2k}$ for every nonnegative integer $k$. Since this sequence is a finite sequence of $1$s and $-1$s that is followed by infinitely many zeros, it corresponds to a finite shooting sequence. The length of $r$ is $\left\lceil\frac{n}{2}\right\rceil$, which is strictly less than $n$ because $n > 1$. Consider the sequence $r^\psi$. By the definition of $r$, the definition of the Thue--Morse morphism, and the property $\big(s^{(0)}\big)_{2k+1} = -\big(s^{(0)}\big)_{2k}$ for every nonnegative $k$, we can see that $r^\psi = s^{(0)}$. 

    If $r$ is a prefix of $T_\infty$, then $s^{(0)} = r^\psi$ is a prefix of $(T_\infty)^\psi = T_\infty$ by Equation \eqref{eq:psithue}. Since $s$ is a prefix of $s^{(0)}$, this implies that $s$ is a prefix of $T_\infty$, violating the conditions of the theorem. Similarly, we can see that $r$ is not a prefix of $-T_\infty$. Since $r$ is a sequence that corresponds to a finite shooting sequence and is neither a prefix of $T_\infty$ nor of $-T_\infty$, the inductive hypothesis implies that $r$ does not have a $(2\ell(r)+1)$-completion. By Proposition~\ref{prop:transform}, this implies that $s^{(0)} = r^\psi$ does not have a $(2\ell(r)+2)$-completion, thereby proving that the $(2n+1)$-completion of $s^{(0)}$ is undefined since $2\ell(r)+2 \leq 2n+1$. Since $s^{(0)}$ does not have a $(2n+1)$-completion, neither does $s$. This completes the proof.
\end{proof}

\section{Final Remarks} \label{sec:FinalRemarks}

Theorem~\ref{thm:main} bounds the least period of the fair extension of a nonexceptional prefix of length $n$ by $4^{n+1}$ (where nonexceptional means that it is not a prefix of either $T_\infty$ or $-T_\infty$); how large can that least period actually be? By Theorem~\ref{thm:fullseq}, the least period is $2^{i+1}$, where $i$ is the smallest index for which the prefix has no $i$-completion, so the question can be reframed as finding the minimum length of a nonexceptional prefix with an $i$-completion. For $i \geq 0$, let $\lambda(i)$ be the smallest length of a nonexceptional sign sequence that has an $i$-completion. 

Since a prefix of $\tilde{s}$ containing $s$ has the same infinite fair extension as $s$, and hence the same least period, the maximum least period among nonexceptional prefixes of length $n$ is exactly $2^{i+2}$ for the largest $i$ with $\lambda(i) \leq n$; conversely, if $\lambda(i) > n$, then every nonexceptional prefix of length $n$ has least period at most $2^{i+1}$. (Such a largest $i$ exists for $n \geq 2$, since $\lambda(0) = 2$ and $\lambda(i) \geq i/2$ by Theorem~\ref{thm:restated}.)

By tracing through the proof of Theorem~\ref{thm:restated}, one can show a stronger bound than the theorem's $\lambda(i) \geq i/2$: namely, for $a \geq 2$, we have $\lambda(2^a-1) \geq 2^a - 3$ and $\lambda(2^a) \geq 2^a - 1$. Since $\lambda(i)$ is nondecreasing, this gives a lower bound that is at least $i/2$ for all $i \geq 4$, with the two bounds being equal precisely when $i$ is 2 below a power of 2.

One natural upper bound for $\lambda(i)$ when $i \geq 2$ is $2^{i-2}+1$: the construction for this is the prefix $T_{i-2} \circ T_0$. This is not a prefix of either $T_\infty$ or $-T_\infty$, since the prefix of $T_\infty$ of length $2^{i-2}+1$ is $T_{i-2} \circ (-T_0)$. For $k < i-2$, the $k$-completion is $T_{i-2} \circ T_{k+1}$. (This can be proved directly, by computing fairness coefficients using Lemma~\ref{lem:concatthue}.) The $(i-2)$-completion is $(T_{i-2})^2 \circ (T_{i-2})^{-2}$, the $(i-1)$-completion is $(T_{i-2})^2 \circ (T_{i-2})^{-2} \circ (T_{i-1})^{-4}$, and the $i$-completion is $(T_{i-2})^2 \circ (T_{i-2})^{-2} \circ (T_{i-1})^{-4} \circ (T_i)^{-12}$. This can be proved by iterating Proposition~\ref{prop:transform} a total of $i-2$ times, using Lemma~\ref{lem:thueexpansions} to apply the morphism to the Thue--Morse blocks. Similarly to Example~\ref{ex:AAcomputation}, the $(i+1)$-completion does not exist.

We can derive an asymptotically better upper bound using the fact that any sequence $s$ that corresponds to a nonzero finite shooting sequence equals its own $i$-completion if and only if $\nu(s) > i$. (This can be proved by induction on $i$.) Applying the multiplication construction of~\cite{BGPW} to a Littlewood polynomial of length $240$ with a root of multiplicity $10$ at $1$ (also from~\cite{BGPW}) yields, for every $i$, a sign sequence that is neither a prefix of $T_\infty$ nor of $-T_\infty$, has fairness exponent greater than $i$, and has length $240^{\lceil (i+1)/10 \rceil}$. This then yields that $\lambda(i) \leq 240^{\lceil (i+1)/10 \rceil}$, which is eventually better than the above construction since $240^{1/10} < 2$.

We believe that the actual growth rate of $\lambda(i)$ is subexponential but superlinear; the following heuristic suggests that it grows as $\Theta(i^3)$. For a nonexceptional prefix $s$, suppose that, conditional on $s^{(i-1)}$ existing, the residue class of $c_i(s^{(i-1)})$ modulo $2^{i(i-1)/2}$ is chosen uniformly at random. If $s^{(i-1)}$ exists, then $b_i^{\,s} = \frac{c_i(s^{(i-1)})}{2^{i(i-1)/2}}$ would be an integer with probability $2^{-i(i-1)/2} = 2^{-\binom{i}{2}}$, so $s^{(i)}$ would exist with that probability. Multiplying over $j \leq i$, the chance that $s$ has an $i$-completion at all would be \[\prod_{j=1}^{i} 2^{-\binom{j}{2}} = 2^{-\binom{i+1}{3}}.\] This implies that the number of nonexceptional sequences of length $n$ with an $i$-completion is roughly $(2^n - 2)2^{-\binom{i+1}{3}}$, so the model would imply that $\lambda(i)$ is roughly the smallest $n$ with $(2^n - 2)2^{-\binom{i+1}{3}} \geq 1$; that is, $\lambda(i) \approx \binom{i+1}{3} = \Theta(i^3)$.

The uniform-coefficient assumption is not literally correct: by Lemma~\ref{lem:2completion}, every sequence has a $2$-completion, whereas the model assigns this a probability of $\frac{1}{2}$. This is caused by modular restrictions on the block counts. This prediction should be treated cautiously for another reason as well: for the Littlewood polynomial question above, which corresponds to finding the shortest nonzero sequence whose $i$-completion is itself, a counting heuristic predicts a length of $\Theta(i^2 \log i)$~\cite{BGPW}. Boyd's superpolynomial lower bound shows that this cannot describe the true asymptotic growth~\cite{Boyd}. Boyd's result does not refute our cubic prediction for $\lambda(i)$, but it illustrates how arithmetic restrictions can invalidate a counting heuristic in a closely related problem. At present we know only that $\lambda(i) = \Omega(i)$ and $\lambda(i) = O(240^{i/10})$; in terms of periods, the maximum least period among nonexceptional prefixes of length $n$ lies between $\Omega(n^{10/\log_2 240}) = \Omega(n^{1.2647\ldots})$ and $2^{O(n)}$. This leads to a natural open question: what is the asymptotic growth of $\lambda(i)$? Our counting heuristic suggests $\lambda(i)=\Theta(i^3)$, corresponding to a maximum least period of $2^{\Theta(n^{1/3})}$, but the influence of the arithmetic restrictions on this prediction remains unclear.

\section*{AI usage}

We used ChatGPT 5.6, Claude Fable 5.1, and ChatGPT 6 to provide feedback on the draft, do literature searches, and clean up the bibliography. We also used Gemini and Perplexity to come up with the notation for early drafts of this paper.

\section*{Acknowledgments}

We are grateful to the MIT PRIMES-USA program for supporting this research and to Andrey Khesin for having a helpful discussion about the topic, as well as sharing a new formulation of the duel problem. We are also grateful to Thomas R\"ud for reviewing our paper and for his helpful suggestions.

\end{document}